\documentclass[reqno, 11pt]{amsart}
\usepackage{amssymb}
\usepackage[numbers]{natbib}
\pdfoutput=1

\usepackage[nesting]{hyperref}
\usepackage[pdftex]{graphicx}
\usepackage{listings}
\usepackage{multirow}
\usepackage{placeins}
\usepackage{color}
\usepackage{subfigure}
\usepackage{lscape}

\newcommand{\BlackBoxes}{\global\overfullrule5pt}

\BlackBoxes

\newcommand{\R}{\mathbb{R}}
\newcommand{\N}{\mathbb{N}}

\newcommand{\E}{\mathbb{E}}

\newcommand{\PP}{\mathbb{P}}
\newcommand{\Eop}{\operatorname{\mathbb{\E}}}

\newcommand{\Pop}{\operatorname{\mathbb{\PP}}}

\renewcommand{\P}{\mathbb{P}}

\newtheorem{theorem}{Theorem}

\newtheorem{lemma}[theorem]{Lemma}

\theoremstyle{definition}
\newtheorem{example}[theorem]{Example}
\newtheorem{remark}[theorem]{Remark}

\numberwithin{equation}{section} \numberwithin{theorem}{section}

\def\0{\kern0pt\-\nobreak\hskip0pt\relax}

\makeatletter
\AtBeginDocument{%
 \def\@serieslogo{%
 \vbox to\headheight{%
 \parindent\z@ \fontsize{6}{7\p@}\selectfont
 \vss}}}

\def\makeoverbar#1#2#3#4#5#6#7{%
 \setbox0=\hbox{$\m@th#2\mkern#5mu{{}#3{}}\mkern#6mu$}%
 \setbox1=\null \dimen@=#4\fontdimen8#13 \dimen@=3.5\dimen@
 \advance\dimen@ by \ht0 \dimen@=-#7\dimen@ \advance\dimen@ by \wd0
 \ht1=\ht0 \dp1=\dp0 \wd1=\dimen@
 \dimen@=\fontdimen8#13 \fontdimen8#13=#4\fontdimen8#13
 \rlap{\hbox to \wd0{$\m@th\hss#2{\overline{\box1}}\mkern#5mu$}}
 \fontdimen8#13=\dimen@}

\def\mylabel#1#2{{\def\@currentlabel{#2}\label{#1}}}

\makeatother

\begin{document}

%\listoftodos

\makeatletter \providecommand\@dotsep{5} \makeatother
%\listoftodos[Changes in Orange/Red To Do List in Green / Blue]\relax

\title[Risk-Sensitive Stopping Problems for Continuous-Time Markov Chains]{Risk-Sensitive Stopping Problems for Continuous-Time Markov Chains}

\author[N. \smash{B\"auerle}]{Nicole B\"auerle${}^*$}
\address[N. B\"auerle]{Institute for Stochastics,
Karlsruhe Institute of Technology, D-76128 Karlsruhe, Germany}

\email{\href{mailto:nicole.baeuerle@kit.edu}
{nicole.baeuerle@kit.edu}}

%\thanks{${}^\ddag$ The underlying projects have been funded by the Bundesministerium f\"ur Bildung und
%Forschung of Germany under promotional reference 03BAPAC1. The
%authors are responsible for the content of this article.}

\author[A. \smash{Popp}]{Anton Popp$^\ddag$}
\address[A. Popp]{Institute for Stochastics,
Karlsruhe Institute of Technology, D-76128 Karlsruhe, Germany}

\email{\href{mailto:anton.popp@kit.edu} {anton.popp@kit.edu}}

\maketitle

\begin{abstract}
In this paper we consider stopping problems for continuous-time Markov chains under a general risk-sensitive optimization criterion for problems with finite and infinite time horizon. More precisely our aim is to maximize the certainty equivalent of the stopping reward minus cost over the time horizon. We derive optimality equations for the value functions and prove the existence of optimal stopping times. The exponential utility is treated as a special case.  In contrast to risk-neutral stopping problems it may be optimal to stop between jumps of the Markov chain. We briefly discuss the influence of the risk sensitivity on the optimal stopping time and consider a special house selling problem as an example.
\end{abstract}

\vspace{0.5cm}
\begin{minipage}{14cm}
{\small
\begin{description}
\item[\rm \textsc{ Key words}]
{\small Markov Decision Problem, Risk-aversion,  Certainty Equivalent, Exponential Utility.}
\item[\rm \textsc{AMS subject classifications}] {\small 60J27,90C40. }
\end{description}
}
\end{minipage}

\section{Introduction}\label{sec:intro}\noindent
In this paper we consider stopping problems for continuous-time Markov chains under a general risk-sensitive optimization criterion for problems with finite and infinite time horizon. More precisely our aim is to maximize the certainty equivalent of the stopping reward over the time horizon. We assume that we have cost as long as we do not stop. The certainty equivalent of a random variable is defined by $U^{-1}(\Eop U(X))$ where $U$ is an increasing concave function. If $U(x)=x$ we obtain as a special case the classical risk-neutral decision maker. The case $U(x)=-e^{-\gamma x }, \gamma>0$ is often referred to as 'risk-sensitive', however the risk-sensitivity is here only expressed in a special way through the risk-sensitivity parameter $\gamma\neq 0$.  More general, the certainty equivalent may be written (assuming enough regularity of $U$) as
\begin{equation}\label{BReq:Urep} U^{-1}\Big(\Eop\big[U(X)\big]\Big)\approx \Eop X - \frac12 l_U(\Eop X) Var[X]\end{equation}
where \begin{equation}\label{BReq:AP}  l_U(x) = -\frac{U''(x)}{U'(x)}\end{equation} is the {\em Arrow-Pratt} function of absolute risk aversion. In case of an exponential utility, this absolute risk aversion $l_U(x)=\gamma$ is constant (for a discussion see \cite{bpli03}). In contrast to the classical risk-neutral situation where a uniformization of the Markov chain immediately leads to the observation that optimal stopping time points can only be jump time points of the continuous-time Markov chain, this is no longer true in our setting with general utility function. We give an explicit example where it is optimal to stop between jumps.

Stopping problems with general utility functions are rarely treated in the literature. We are only aware of some papers considering the problem in discrete time. \cite{m00} considers the classical house selling problem with general utility in a discrete time setting. In a separate section we treat a continuous time version of the house selling problem. We show that some of the results in \cite{m00} also extend to our case but in general the problem is different. \cite{Kadota1996,Kadota2001} consider stopping problems with denumerable state space and arbitrary utility function. The authors there discuss the so-called {\em monotone case} and give conditions for the optimality of {\em one-step-look-ahead rules}. In \cite{br15} risk-sensitive stopping problems with general utility are considered in a partially observable setting. Optimality equations, examples and risk-sensitivity results are considered there.

Of course the stopping problems we treat here can be seen as a special case of risk-sensitive continuous-time Markov Decision Processes. The theory for these type of problems with an exponential utility has been treated in \cite{Ghosh2014}. There both finite and infinite time horizon problems are considered and the value function is characterized via the HJB equation and an optimal Markov control is obtained. The infinite horizon average cost case is also considered. In \cite{wei} the author studies continuous-time Markov decision processes under the risk-sensitive finite-horizon cost criterion with the exponential utility. Suitable optimality conditions are given and a Feynman Kac formula is established, via which the existence and uniqueness of the solution to the optimality equation and the existence of an optimal deterministic Markov policy are obtained. However, in our paper we will see that the exponential utility case is always special and often behaves as the risk-neutral case.

Risk-sensitive Markov Decision Processes in discrete-time with arbitrary utility functions have been considered in \cite{br11}. There optimality equations for finite and infinite time horizon problems can be found as well as results about the existence of optimal policies. For a specific application to a dividend problem see \cite{bj14}.

The paper is organized as follows: First we will introduce the risk-sensitive stopping problem together with some integrability and regularity assumptions. Then we characterize the feasible stopping times which leads to a formulation with the help of decision rules and which allows a recursive solution. Then we consider
risk-sensitive stopping problems with a finite time horizon. By finite time horizon we mean that one latest has to stop after the $n$-th jump. We consider both problems where the utility function has domain $\R$ like e.g. in the exponential case and where the utility function has restricted domain, like e.g. $U(x)=\sqrt{x}$ or $U(x)=\ln(x)$.  We derive a recursive algorithm to compute the value function and the optimal stopping time. An example with logarithmic utility shows that it may be optimal to stop between jumps. In the case of an exponential utility function however the optimality equation simplifies and it is possible to show that optimal stopping times are restricted to the jump time points of the continuous-time Markov chain. In section 5 we consider the  risk-sensitive stopping problem with infinite time horizon. We show that the value function satisfies a fixed point equation and give conditions under which a maximizer of this equation defines an optimal stopping time. Again results simplify in the exponential utility case. Then we give sufficient conditions in the general utility case which imply that it is optimal to stop directly after a jump. These conditions can be interpreted as one-step look ahead rules in the case of an exponential utility. In section 7 we shortly discuss the influence of risk aversion on the optimal stopping time. It will turn out that more risk averse decision makers will not stop earlier. Finally in the last section we will consider a special house selling problem where we can show a monotonicity property of the optimal stopping time.

\section{Risk-Sensitive Stopping Problems}\label{sec:intro}\noindent
We suppose that a {\em continuous-time Markov chain} $(X_t)$ with countable state space $S$ and intensity matrix $Q=(q_{ij})_{i,j\in S}$ is given. For simplicity it is assumed that $$0<  q_i := -q_{ii} = \sum_{j\neq i} q_{ij} < \infty$$ i.e. the Markov chain is conservative and has no absorbing states. The underlying probability space is $(\Omega,\mathcal{F},\Pop)$. Trajectories are assumed to be right-continuous. We denote by $0=: S_0< S_1 < S_2<\ldots$ the random jump time points of the Markov chain and by $(Z_n)$ the embedded process, i.e. $Z_n = X_{S_n}$. Thus, we can represent the Markov chain by
$$ X_t = \sum_{k=0}^\infty Z_k \cdot 1_{\{S_k \le t < S_{k+1}\}},\quad t\ge 0.$$ In particular $Z_0=X_0$. The natural filtration which is generated by this process is denoted by $(\mathcal{F}_t^X)$ with $\mathcal{F}_t^X := \sigma\big(X_s, s\le t\big)$. It is well-known that $S_{k+1}-S_k \sim\exp(q_{Z_k})$ and that the transition probabilities for the embedded Markov chain are given by $$\Pop(Z_{k+1}=j | Z_k=i)=\frac{q_{ij}}{q_i}$$ for $j\neq i$ and $\Pop(Z_{k+1}=i | Z_k=i)=0$ (see e.g. \cite{Bremaud1999}).

Next suppose a utility function $U : dom(U)\to \R$ is given, i.e. $U$ is strictly increasing, strictly concave and $dom(U)=[d,\infty)$, $dom(U) = (d,\infty)$ or $dom(U)=\R$ where $dom(U)$ denotes the domain of the utility function and $d\in\R$ is a constant. We can extend $U$ on $\R$ by setting
$$ \hat{U}(x) := \left\{\begin{array}{cl}
   -\infty & , x\notin dom(U)\\
   U(x) &, x\in dom(U).
   \end{array}\right.
   $$
For simplicity we will still denote this function by $U$. Next, there is a measurable reward function $g:S\to \R$ and a cost rate $c>0$. We denote by
$$ \Sigma := \{\tau :  \Omega\to [0,\infty) \; |\; \tau \mbox{ is an } (\mathcal{F}_t^X)-\mbox{stopping time with } \Pop_i(\tau<\infty)=1, \mbox{ for } i\in S\}$$
where $\Pop_i(\cdot)$ is the conditional probability measure given $X_0=i$. The aim is to solve the stopping problem
\begin{equation}\label{eq:problem} \sup_{\tau \in \Sigma} \Eop_i\big[ U\big(g(X_\tau)-c\tau\big)\big].\end{equation}
In order to obtain a well-defined problem we make the following assumptions:
\begin{itemize}
  \item[(A1)] $\sup_{\tau \in \Sigma} \Eop_i\big[g^+(X_\tau)-c\tau\big] < \infty,$ $i\in S$.
  \item[(A2)] $\liminf_{n\to\infty} \Eop_i\big[ U\big(g(X_{\tau\wedge S_n})-c(t+\tau\wedge S_n)\big)\big]\ge \Eop_i\big[ U\big(g(X_{\tau})-c(t+\tau)\big)\big]$ for all $i\in S, t\ge 0$ and $\tau\in\Sigma$.
\end{itemize}

\section{Characterization of Stopping-Times}
Before we tackle the stopping problem, let us consider in more detail the stopping times. It turns out that stopping times in $\Sigma$ can be decomposed into a sequence of measurable mappings. This observation has already been used in similar settings by  \cite{davis93} and \cite{Pham2010}.
%In what follows we denote by $(X_t^{S_n})$ the Markov Chain which is stopped after the $n$-th jump, i.e. $X_t^{S_n}:= X_{t\wedge S_n}$. These stopped processes then generate the filtrations $(\mathcal{F}_t^{n,X})$ with $\mathcal{F}_t^{n,X} := \sigma\big(X_s^{S_n}, 0\le s\le t\big)$. We also set $\mathcal{F}_t^{\infty,X}:= \mathcal{F}_t^{X}$.
The following theorem can be seen as a special case of Theorem 2.1 in \cite{Bayraktar2014}:

\begin{theorem}\label{thm:decomposition-n-step}
Let $\tau :\Omega\to[0,\infty)$ be a measurable mapping with $\Pop_i(\tau <\infty)=1$ for $i\in S$. Then $\tau$ is an $(\mathcal{F}^{X}_t)$-stopping time, if and only if it has the following decomposition:
\begin{equation}\label{eq:decomposition-n-step}
   \tau = \tau^0 1_{\{\tau < S_1\}} + \sum_{k=1}^{\infty} \tau^k 1_{\{S_k\le \tau < S_{k+1}\}},\quad \Pop-a.s.
\end{equation}
where for every $k\in\N_0$:
\begin{enumerate}
\item[(i)]   $\tau^k \ge S_k$,
\item[(ii)]  %$\tau^k$ is an $(\cF^{k,X}_t)_{t\ge 0}$--stopping time,
 there exists a measurable mapping $h_k : [0,\infty)^{k+1} \times S^{k+1} \to [0,\infty]$, such that $h_k \ge 0$ and
                 \begin{equation}\label{eq:decomposition-stopping-rule-measurable-mapping}
                   \tau^k = h_k (S_0,\dots,S_k,Z_0,\dots,Z_k) + S_k.
                 \end{equation}
  \end{enumerate}
This decomposition \eqref{eq:decomposition-n-step} is unique in the sense that every term in the sum of \eqref{eq:decomposition-n-step} is $\Pop$-a.s. uniquely determined on the set $\{S_k\le\tau <S_{k+1}\}$.
\end{theorem}
In the next sections we will restrict to {\em Markovian stopping times}. By Markovian we mean that the functions $h_k(S_1,\dots,S_k,Z_0,\dots,Z_k)$ in the decomposition depend only on the current state of the Markov chain and the total time elapsed so far, i.e.  $h_k(S_k,Z_k)$. We denote this class of $(\mathcal{F}^{X}_t)$-stopping times by $\Sigma^M$. This assumption is made to ease the presentation. Indeed it can be shown that the optimal stopping time for problem \eqref{eq:problem} can be found among the Markovian stopping times (for more details see \cite{popp16}). In what follows we will identify $\tau\in \Sigma^M$ with the sequence $\tau=(h_0,h_1,\ldots)$ with measurable $h_k : [0,\infty)\times S\to \R_+$.

\section{Finite Horizon Problems}\label{sec:finiteh}
In this section we will first consider stopping problems with a finite time horizon. By finite time horizon we mean that one has to stop  latest at time $S_n$ when the $n$-th jump occurs. Moreover, we assume now that the process already has a 'history' of $t$ time units where we did not stop, i.e for $\tau \in \Sigma^M$ and $(t,i)\in[0,\infty)\times S$ let
\begin{eqnarray}
\nonumber  V_n(t,i,\tau)  &:=& \Eop_i\Big[ U\big(g(X_{\tau\wedge S_n})-c(t+\tau\wedge S_n)\big)\Big],  \\
\label{eq:problemfinite}  V_n(t,i) &:=& \sup_{\tau \in \Sigma^M} V_n(t,i,\tau).
\end{eqnarray}
Here we interpret stopping times $\tau=(h_0,h_1,\ldots)\in\Sigma^M$ as in \eqref{eq:decomposition-n-step} with $\tau^k = h_k(t+S_k,Z_k)+S_k, k\in\N_0$. In particular $\tau^0=h_0(t,Z_0)$.
Due to assumption (A1) $V_n(t,i)<\infty$ is well-defined because a utility function can be bounded from above by a linear function. Moreover, it follows directly from the monotonicity of $U$ that $t\mapsto V_n(t,i)$ is decreasing for all $i\in S$ and all $n\in\N$. We are interested in finding $V_n(0,i)$.

\subsection{Reward Iteration}
In this section we note that for a given stopping time $\tau \in \Sigma^M$, the corresponding value $V_n(t,i,\tau)$ can be computed recursively. In order to formulate this statement let for $\tau\in\Sigma^M$ with $\tau=(h_0,h_1,\ldots)$ the stopping time $\overrightarrow{\tau}$ be defined by $\overrightarrow{\tau}=(h_1,h_2,\ldots)\in\Sigma^M$. Then we obtain:

\begin{theorem}\label{thm:reward-iteration}
Let $\tau \in \Sigma^M$ with $ \tau=(h_0,h_1,\ldots)$. We have $ V_0(t,i,\tau) = U\big(g(i)-ct \big)$ and the following {\em reward iteration} holds for $k=0,\ldots n-1$:
  \begin{equation}\label{eq:reward-iteration}
    V_{k+1}(t,i,\tau) = U\big(g(i)-c(t+h_0(t,i)) \big) e^{-q_i h_0(t,i)}
           + \int_0^{h_0(t,i)}e^{-q_i s} \sum_{j\neq i} q_{ij} V_k(t+s,j,\overrightarrow{\tau}) \;ds.
  \end{equation}
\end{theorem}

\begin{proof}
For $n=0$ the statement follows directly from the definition since $S_0=0$. For $k+1$ we obtain with the Markov property of $(X_t)$:
\begin{align*}
  V_{k+1}(t,i,\tau) &= \Eop_i\left[ U\left( g(X_{\tau\wedge S_{k+1}})-c(t+\tau\wedge S_{k+1})  \right)\right]\\
                      &= \Eop_i\left[ U\left(g(X_{\tau\wedge S_{k+1}}) -c(t+\tau\wedge S_{k+1})\right) 1_{\{S_1>\tau\}}\right]\\
                      &\quad + \Eop_i\left[ U\left(g(X_{\tau\wedge S_{k+1}}) -c(t+\tau\wedge S_{k+1})\right) 1_{\{S_1\le\tau\}}\right]\\
                       &= \Eop_i\left[ U\left(g(i) -c(t+h_0(t,i))\right) 1_{\{S_1> h_0(t,i)\}}\right]\\
                       &\quad + \int_0^{h_0(t,i)}q_i e^{-q_i s} \sum_{j\neq i} \frac{q_{ij}}{q_i} \Eop_i\left[ U\left(g(X_{{\tau}\wedge S_{k+1}}) -c(t+{\tau}\wedge S_{k+1})\right) \Big| S_1=s, Z_1=j\right] \; ds\\
                       &=  U\left(g(i) -c(t+h_0(t,i))\right) \Eop_i\left[1_{\{S_1> h_0(t,i)\}}\right]\\
                      &\quad + \int_0^{h_0(t,i)}e^{-q_i s} \sum_{j\neq i} q_{ij} \Eop_j\left[ U\left(g(X_{\overrightarrow{\tau}\wedge S_k}) -c(t+s+\overrightarrow{\tau}\wedge S_k)\right) \right] \; ds\\
                      &=  U\left(g(i) -ct -c\; h_0(t,i)\right) e^{-q_i h_0(t,i)} + \int_0^{h_0(t,i)}e^{-q_i s} \sum_{j\neq i} q_{ij} V_k(t+s,j,\overrightarrow{\tau}) \;ds.
  \end{align*}
Also note that here
\begin{equation*}
   \overrightarrow{\tau} = \overrightarrow{\tau}^0 1_{\{\tau < S_1\}} + \sum_{k=1}^{\infty} \overrightarrow{\tau}^k 1_{\{S_k\le \tau < S_{k+1}\}},\quad \Pop-a.s.
\end{equation*}
with $\overrightarrow{\tau}^k=h_k(t+s+S_k,Z_k)+S_k$ is a stopping time which starts from scratch at time $s$. This implies that the statement is true for $k+1$.
\end{proof}

Let $\mathbb{M} := \{ v:[0,\infty)\times S\to\R \:\cup\: \{-\infty\} \; |\; v \mbox{ is measurable} \}$.
Next define the following $\mathbf{T}$-operator which is defined on $\mathbb{M}$ and returns a function $(Tv):[0,\infty)\times S\to\R\cup\{-\infty\}$:
$$ (\mathbf{T}v)(t,i) := \sup_{\vartheta\ge 0} \Big\{  U\left(g(i) -c(t+\vartheta)\right) e^{-q_i \vartheta} + \int_0^{\vartheta}e^{-q_i s} \sum_{j\neq i} q_{ij} v(t+s,j) \;ds\Big\}.$$
We now have to distinguish whether $U$ has domain $\R$ (which is true e.g. for $U(x)=-e^{-\gamma x}$) or whether the domain of $U$ is restricted (which is true e.g. for $U(x)=\sqrt{x}$ or $U(x)=\ln x$).

\subsection{The utility function is defined on $\R$}
Here we will be more precise about the domain and image of the  $\mathbf{T}$-operator. Since $U$ is concave, $U$ is bounded from above by a linear function.  Suppose that $U(x)\le ax+b$ for $a,b\in\R_+$. Note that $U$ is continuous on $\R$. Let us define
\begin{eqnarray*}
  \mathbb{M}_n &:=& \big\{ v\in \mathbb{M} \; |\;  v(t,i) \le a\sup_{\tau \in \Sigma^M} \Eop_i\big[g^+(X_{\tau\wedge S_n})-c(t+\tau\wedge S_n)\big] +b,\\
   &&  \quad\quad t\mapsto v(t,i) \mbox{ is decreasing  for all } i\in S\big\}.
\end{eqnarray*}
Then it is possible to show:

\begin{lemma}\label{lem:M1}
It holds that $\mathbf{T} : \mathbb{M}_n \to \mathbb{M}_{n+1}$ for $n\in\N$. Moreover, the exists a measurable $h:[0,\infty)\times S\to \R_+\cup\{\infty\}$ s.t.
$$ (\mathbf{T} v)(t,i) = U\left(g(i) -c(t+h(t,i))\right) e^{-q_i h(t,i)} + \int_0^{h(t,i)}e^{-q_i s} \sum_{j\neq i} q_{ij} v(t+s,j) \;ds.$$
In this case we call $h$ a maximizer of $\mathbf{T} v$.
\end{lemma}

\begin{proof}
Let $v\in \mathbb{M}_n$. First we show the upper bound: Since $v\in \mathbb{M}_n$ there exists for all $\varepsilon>0$ a stopping time $\tau^\varepsilon$ s.t. $$v(t,i) \le a \Eop_i\big[g^+(X_{\tau^\varepsilon\wedge S_n})-c(t+\tau^\varepsilon\wedge S_n)\big] +b+\varepsilon.$$ Thus we obtain for all $\vartheta\ge 0$ like in the proof of Theorem \ref{thm:reward-iteration}:
\begin{eqnarray*}
   && U\left(g(i)-c(t+\vartheta)\right) e^{-q_i \vartheta} + \int_0^{\vartheta}e^{-q_i s} \sum_{j\neq i} q_{ij} v(t+s,j) \;ds \\
   &\le& \big( a(g^+(i) -c(t+\vartheta))+b\big) e^{-q_i \vartheta} +\\
   && +  \int_0^{\vartheta}e^{-q_i s} \sum_{j\neq i} q_{ij} \big( a\Eop_j\big[g^+(X_{\tau^\varepsilon\wedge S_n})-c(t+s+\tau^\varepsilon\wedge S_n)\big] +b\big) \;ds+\varepsilon \\
   && \le a \Eop_i\big[g^+(X_{\sigma^\varepsilon\wedge S_{n+1}})-c(t+\sigma^\varepsilon\wedge S_{n+1})\big] +b+\varepsilon
\end{eqnarray*}
where $\sigma^\varepsilon=(h_0,\tau^\varepsilon)$ with $h_0\equiv \vartheta$.
Since this is true for all $\varepsilon,\vartheta\ge 0$ we obtain the upper bound by letting $\varepsilon\downarrow 0$ and by taking the supremum over all stopping times $\sigma^\varepsilon$.

Next $t\mapsto (\mathbf{T} v)(t,i)$ is decreasing since $t\mapsto U(t)$ is increasing and $t\mapsto v(t,i)$ is by assumption decreasing.

Last but not least we show that $\mathbf{T} v$ is again measurable and there exists a measurable selector. Since $S$ is discrete, we can concentrate on $t$. The first part $(t,\vartheta)\mapsto U\left(g(i) -ct -c\vartheta\right) e^{-q_i \vartheta}$ is even continuous by our assumptions on $U$. For the second part  $(t,\vartheta)\mapsto \int_0^{\vartheta}e^{-q_i s} \sum_{j\neq i} q_{ij} v(t+s,j) \;ds$ is measurable and continuous in $\vartheta$ since by assumption $s\mapsto v(t+s,j)$ is decreasing and can thus only have a countable number of jumps on $[0,\vartheta]$. Let us define
$$m(t,\vartheta) := U\left(g(i)-c(t+\vartheta)\right) e^{-q_i \vartheta} + \int_0^{\vartheta}e^{-q_i s} \sum_{j\neq i} q_{ij} v(t+s,j) \;ds.$$
We can now apply the measurable selection theorem of \cite{Brown1973} (Corollary 1) which states that on $I := \{ t\in\R_+ :  m(t,\vartheta^*)= \sup_{\vartheta\ge 0} m(t,\vartheta) \mbox{ for some } \vartheta^*\in\R_+\}$ there exists a measurable selector $\varphi$ s.t. $m(t,\varphi(t)) = \sup_{\vartheta\ge 0} m(t,\vartheta)$. Thus $m(t,\varphi(t))$ is again measurable. Outside $I$ we have $$ \sup_{\vartheta\ge 0} m(t,\vartheta) = \lim_{\vartheta\to \infty} m(t,\vartheta)$$
which is measurable as a limit of measurable functions.
\end{proof}

\subsection{The utility function is defined on a subset of $\R$}
Now we assume that $dom(U)=[d,\infty)$ with $d<\inf_{i\in S} g(i)$  and that $U$ is continuous on its domain. Then obviously the domain of $V_0(\cdot,i)$ is given by $[0,\frac{g(i)-d}{c}]$ when the initial state of the Markov chain is $i$.
%Let us recursively define the time points
%\begin{eqnarray*}
%  T_0(i)  &:=& \frac{g(i)-d}{c} \\
%  T_n(i)  &:=&\min\Big\{ \inf_{\substack{j\neq i,\\ q_{ij}>0}} T_{n-1} (j), \frac{g(i)-d}{c}\Big\}.
%\end{eqnarray*}
Let us define $d(i) := \frac{g(i)-d}{c}$ and
\begin{eqnarray*}
  \mathbb{M}_n &:=& \big\{ v(\cdot,i): [0,d(i)] \to \R\; |\, v \mbox{ is decreasing  and continuous for all } i\in S,\\
   &&\quad\quad v(t,i) \le a\sup_{\tau \in \Sigma^M} \Eop_i\big[g^+(X_{\tau\wedge S_n})-c(t+\tau\wedge S_n)\big] +b\big\}.
\end{eqnarray*}
Then it is possible to show:

\begin{lemma}\label{lem:M2}
It holds that $\mathbf{T} : \mathbb{M}_n \to \mathbb{M}_{n+1}$ for $n\in\N$. Moreover, the exists a measurable $h(\cdot,i):[0,d(i)]\to \R^+$ s.t.
$$ (\mathbf{T} v)(t,i) = U\left(g(i) -ct -c h(t,i)\right) e^{-q_i h(t,i)} + \int_0^{h(t,i)}e^{-q_i s} \sum_{j\neq i} q_{ij} v(t+s,j) \;ds$$
i.e. $h$ is a maximizer of $\mathbf{T} v$.
\end{lemma}

\begin{proof}
Fix $i\in S$. The upper bound follows in the same way as in the proof of Lemma \ref{lem:M1}. Let us next consider the domain. Suppose $v\in \mathbb{M}_n$ with domain $[0,\frac{g(i)-d}{c}]$  and consider $$m(t,\vartheta):=U\left(g(i) -c(t+\vartheta)\right) e^{-q_i \vartheta} + \int_0^{\vartheta}e^{-q_i s} \sum_{j\neq i} q_{ij} v(t+s,j) \;ds.$$
When we set $\vartheta=0$, then $m(t,0)\in\R$  if and only if $t\le  \frac{g(i)-d}{c}$. If $\vartheta>0$ then the interval on which $m(t,\vartheta)$ is finite can only get smaller. Hence the domain of $t\mapsto (\mathbf{T} v)(t,i)$ is again $[0,d(i)]$.

Next $t\mapsto (\mathbf{T} v)(t,i)$ is decreasing since $t\mapsto U(t)$ is increasing and $t\mapsto v(t,i)$ is by assumption decreasing.

Finally we have to show that $t\mapsto (\mathbf{T} v)(t,i)$ is continuous and the existence of a maximizer. But this follows from Theorem 2.4.10 in \cite{br11} since $(t,\vartheta) \mapsto m(t,\vartheta)$ is continuous, the set $[0,\frac{g(i)-ct-d}{c} ]$ over which the function has to be maximized is compact, and the set-valued mapping $t\mapsto [0,\frac{g(i)-ct-d}{c} ]$ is continuous.
\end{proof}

\begin{remark}
The case that $dom(U)=(d,\infty)$ with $d<\inf_{i\in S} g(i)$  and $U$ is continuous on its domain with $\lim_{x\downarrow d} U(x)=-\infty$ can be treated similarly. Here we have to consider
\begin{eqnarray*}
  \mathbb{M}_n &:=& \big\{ v(\cdot,i): [0,d(i)) \to \R \; |\; v \mbox{ is decreasing  and continuous for all } i\in S,\\
   &&\quad\quad v(t,i) \le a\sup_{\tau \in \Sigma^M} \Eop_i\big[g^+(X_{\tau\wedge S_n})-ct-c(\tau\wedge S_n)\big] +b\big\}.
\end{eqnarray*}
Then Lemma \ref{lem:M2} holds in analogous way. The existence of a maximizer follows by considering the level sets $\{\vartheta \ge 0 : m(t,\vartheta)\ge m(t,0)\}$ for optimization which are again compact.
\end{remark}

\subsection{The optimality equation}
Combining the results of the previous subsections we obtain in both cases the following result (where in the case of bounded domain we set $V_n(t,i)=-\infty$ if $t$ is not in the domain).

\begin{theorem}\label{thm:bellman}
  \begin{itemize}
  \item[a)]  For  $(t,i)\in [0,\infty)\times S$ it holds that $ V_0(t,i) = U\big(g(i)-ct\big)$ and for $k=0,1\ldots ,n-1$
          \begin{align}
         \nonumber   V_{k+1}(t,i) =& (\mathbf{T} V_k)(t,i)\\ \label{eq:bellman-equation}
            & \sup_{\vartheta\ge 0}\bigg\{ U\big(g(i)-c(t+\vartheta) \big)  e^{-q_i\vartheta} + \int_0^{\vartheta}e^{-q_i s} \sum_{j\neq i} q_{ij} V_k(t+s,j) \;ds \bigg\}.
          \end{align}
  \item[b)]For each $k=0,1,\ldots,n-1$ there exist maximizers $h_k^*$ of $\mathbf{T} V_k$ and the stopping time defined by
  $ \tau^*=(h_0^*,h_1^*,\ldots,h_n^*,\ldots)$ is optimal for problem \eqref{eq:problemfinite}.
  \end{itemize}
\end{theorem}

The proof follows from Theorem 2.3.8 in \cite{br11} since Lemma \ref{lem:M1} and Lemma \ref{lem:M2} respectively show that the structure assumption is satisfied.

The interesting observation for these risk-sensitive stopping problems is the fact that it might be optimal to stop between jumps of the Markov chain. This is in contrast to risk-neutral stopping problems where it is a folk theorem that it is enough to consider only jump time points for optimal stopping. The next example highlights this fact.

\begin{example}
 Let $S=\{0,1\}$ and let $(X_t)$ be a continuous-time Markov chain with intensity matrix $Q$ given by
   $$      Q =  \begin{pmatrix} -\alpha & \alpha \\ \beta & -\beta \end{pmatrix}$$
   for some $\alpha,\beta>0$. We consider the logarithmic utility function $U(x)=\ln (x)$. Moreover, let $c>0$ be the cost rate and assume that $g(0)>0$ is the gain when we stop in state $0$ and $ g(1) = K g(0)$ is the gain when we stop in state $1$ where $K>1$. It is not difficult to see that the domain of $V_k(\cdot,0)$ for all $k$ is given by $[0,\frac{g(0)}{c})$. Moreover, we obtain for all $k$
   $$   V_k(t,1) = U\big(g(1)-ct\big) = \ln\big(g(1)-ct\big),\quad t\in \Big[0,\frac{g(1)}{c}\Big) $$
    and it is optimal to stop immediately, i.e. $\tau^*\equiv 0.$ Now consider state $0$. From the optimality equation we get
    \begin{align*}
     V_{k+1}(t,0) = & \sup_{\vartheta\ge 0} \Big\{ U\big(g(0)-c(t+\vartheta)\big)  e^{-q_0\vartheta}
                + \int_0^{\vartheta}e^{-q_0 s} q_{01} V_k(t+s,1) \;ds\Big\}\\
              = & \sup_{\vartheta\ge 0} \Big\{ \ln\big(g(0)-c(t+\vartheta) \big) e^{-\alpha \vartheta}
                + \int_0^{\vartheta}e^{-\alpha s}  q_{01} \ln\big(g(1)-c(t+s)\big) \;ds \Big\}.
          \end{align*}
Differentiating this function we see that the maximum point $\vartheta^*$ is either the unique solution $\vartheta$ of the equation
\begin{equation*}\label{eq:exa-bellman-log-utility-zero-condition}
             e = \left[ 1+\frac{\frac{\alpha}{c}\big(g(1)-g(0)\big)}{\alpha\left( \frac{g(0)}{c}-(t+\vartheta)\right)}\right]^{\alpha\left( \frac{g(0)}{c}-(t+\vartheta)\right)}
           \end{equation*}
whenever this point is in    $[0,\frac{g(0)}{c})$. Otherwise $\vartheta^*=0$. When we consider the specific values
$$ g(0)=10, \quad g(1)=\frac{e^{10}+99}{10} \approx 2212.55 \quad \text{ and } \quad \alpha=\beta=c=1$$
then we obtain for the optimal stopping time $\tau^*=(h^*,h^*,\ldots)$ with
$$  h^*(t,i) = \begin{cases}
                              0,                    &i=1 \quad\text{ or }\quad i=0 \text{ and } t\ge 9.9,\\
                              9.9-t,  &i=0 \text{ and } t\in[0, 9.9).
                            \end{cases}$$
This means in state  $(t,0)$ we are willing to wait for a jump into the 'good' state $1$ but only for the limited amount of time $9.9-t$. If this time is over we will stop in the 'bad' state $0$.
\end{example}

The example also shows that the optimal stopping time $\tau^*=(h_0^*,h_1^*,\ldots)$ satisfies a certain consistency condition. A result which we will generalize in the next lemma:

\begin{lemma}\label{lem:consist}
Let $h_k^*$ be a maximizer of \eqref{eq:bellman-equation} and suppose that $h_k^*(t,i) \ge \delta>0$. Then \linebreak $h_k^*(t+\delta,i)=h_k^*(t,i)-\delta$.
\end{lemma}

\begin{proof}
Fix $t\ge 0$. We have to maximize
$$ \vartheta\mapsto m(t,\vartheta):=  U\big(g(i)-c(t+\vartheta) \big)  e^{-q_i\vartheta} + \int_0^{\vartheta}e^{-q_i s} \sum_{j\neq i} q_{ij} V_k(t+s,j) \;ds $$
By an obvious substitution in the integral we obtain
$$ m(t,\vartheta)=  e^{q_i t} \Big( U\big(g(i)-c(t+\vartheta) \big)  e^{-q_i(t+\vartheta)} + \int_t^{t+\vartheta}e^{-q_i \hat{s}} \sum_{j\neq i} q_{ij} V_k(\hat{s},j) \;d\hat{s}\Big). $$
Since $e^{q_i t}>0$, maximizing $m(t,\vartheta)$ in $\vartheta$ leads to the same maximum point than maximizing
$$ \vartheta\mapsto U\big(g(i)-c(t+\vartheta) \big)  e^{-q_i(t+\vartheta)} + \int_t^{t+\vartheta}e^{-q_i \hat{s}} \sum_{j\neq i} q_{ij} V_k(\hat{s},j) \;d\hat{s}.$$
Adding now the constant $\int_0^{t}e^{-q_i \hat{s}} \sum_{j\neq i} q_{ij} V_k(\hat{s},j) \;d\hat{s}$  does not change the maximum point, hence we can equivalently maximize
$$ \vartheta\mapsto U\big(g(i)-c(t+\vartheta) \big)  e^{-q_i(t+\vartheta)} + \int_0^{t+\vartheta}e^{-q_i \hat{s}} \sum_{j\neq i} q_{ij} V_k(\hat{s},j) \;d\hat{s}.$$
The resulting function however depends on $(t,\vartheta)$ only by $t+\vartheta$. This implies the result.
\end{proof}

\subsection{Exponential utility}
Let us now consider the special case $U(x)=-e^{-\gamma x}$ for $\gamma>0$ and $x\in\R$.  Obviously $V_0(t,i)= -e^{c\gamma t} e^{-\gamma g(i)}$. By definition of the value functions in \eqref{eq:problemfinite} we have in this special case $V_k(t,i)=e^{c\gamma t} W_k(i)$ with a function $W:S\to (-\infty,0]$.  Here we obtain

\begin{theorem} If $q_i\le c\gamma$ we obtain that $W_k(i)= -e^{-\gamma g(i)}$ for $k=0,1\ldots ,n-1$ and $\tau^*\equiv 0$ is optimal. Now suppose that $q_i > c\gamma$. Then it holds:
  \begin{itemize}
  \item[a)]  For  $(t,i)\in [0,\infty)\times S$ it holds that $ W_0(i) = -e^{-\gamma g(i)}$ and for $k=0,1\ldots ,n-1$
          \begin{align}\label{eq:bellman-equation_exp1}
            W_{k+1}(i) =& \max\bigg\{-e^{-\gamma g(i)}, \sum_{j\neq i} \frac{q_{ij}}{q_i-c\gamma} W_k(j) \bigg\}.
          \end{align}
  \item[b)]For each $k=0,1,\ldots,n-1$ there exist maximizers $f_k^*:S\to \{0,\infty\}$ (where we set $f_k^*(i):=0$ if the maximum is attained in the first expression and $f_k^*(i):=\infty$ if the maximum in \eqref{eq:bellman-equation_exp1} is attained in the second expression) and the stopping time defined by
\begin{equation*}
   \tau^* = f_0^*(Z_0) 1_{\{\tau^* < S_1\}} + \sum_{k=1}^{\infty} (f_k^*(Z_k)+S_k) 1_{\{S_k\le \tau^* < S_{k+1}\}},\quad \Pop-a.s.
\end{equation*}
 is optimal for the problem with exponential utility.
  \end{itemize}
\end{theorem}

Note that this result in particular implies that for the exponential utility it is never optimal to stop between jump times of the Markov chain like in the risk-neutral case.

\begin{proof}
The proof is by induction on $k$. Note that for $k=0$ the statement is obvious. Suppose the statement is true for $k$. By Theorem \ref{thm:bellman} it holds that
\begin{eqnarray*}
  V_{k+1}(t,i) &=&  \sup_{\vartheta\ge 0}\bigg\{ U\big(g(i)-c(t+\vartheta) \big)  e^{-q_i\vartheta} + \int_0^{\vartheta}e^{-q_i s} \sum_{j\neq i} q_{ij} V_k(t+s,j) \;ds \bigg\}\\
   &=&   e^{\gamma ct} \sup_{\vartheta\ge 0}\bigg\{ -e^{-\gamma g(i)} e^{\vartheta (\gamma c-q_i)} + \int_0^{\vartheta}e^{s(\gamma c-q_i )} \sum_{j\neq i} q_{ij} W_k(j) \;ds \bigg\}\\
   &=&  e^{\gamma ct} \sup_{\vartheta\ge 0}\bigg\{ -e^{-\gamma g(i)} e^{\vartheta (\gamma c-q_i)} + (1-e^{\vartheta(\gamma c-q_i)}) \sum_{j\neq i} \frac{q_{ij}}{q_i-\gamma c} W_k(j) \bigg\}.
\end{eqnarray*}
For $q_i \le c\gamma$ the expression
$$ \vartheta \mapsto -e^{-\gamma g(i)} e^{\vartheta (\gamma c-q_i)} + (1-e^{\vartheta(\gamma c-q_i)}) \sum_{j\neq i} \frac{q_{ij}}{q_i-\gamma c} W_k(j)$$
is decreasing (note that $W_k\le 0$) and thus $\vartheta^*=0$ is the maximum point.
For $q_i > c\gamma$ the expression is a convex combination
$$-e^{-\gamma g(i)} \alpha_i(\vartheta) + (1-\alpha_i(\vartheta)) \sum_{j\neq i} \frac{q_{ij}}{q_i-\gamma c} W_k(j)$$
with $\alpha_i(\vartheta)\in(0,1)$. Thus we obtain:
 $$  V_{k+1}(t,i)= e^{\gamma ct} W_{k+1}(i) =e^{\gamma ct} \max\Big\{ -e^{-\gamma g(i)}, \sum_{j\neq i} \frac{q_{ij}}{q_i-\gamma c} W_k(j)\Big\}$$
 which implies the result.
\end{proof}

\section{Infinite Horizon Problems}\label{sec:infiniteh}
Let us now consider the optimization problem \eqref{eq:problem} with unrestricted time horizon. First note that $V_n$ is increasing in $n$ since the stopping times which are available for $V_n$ are also available for $V_{n+1}$. The following theorem is valid for all domains of $U$ with the interpretation that $V(t,i)=-\infty$ if $t$ is not in the domain.

\begin{theorem}
The value function $V(t,i)$ of \eqref{eq:problem} satisfies for $(t,i)\in[0,\infty)\times S$
\begin{align}
   \nonumber         V(t,i) =&  (\mathrm{T}V)(t,i)\\ \label{eq:fixed point}
            =&\sup_{\vartheta\ge 0}\bigg\{ U\big(g(i)-c(t+\vartheta) \big)  e^{-q_i\vartheta} + \int_0^{\vartheta}e^{-q_i s} \sum_{j\neq i} q_{ij} V(t+s,j) \;ds \bigg\}.
          \end{align}
Moreover, there exists a maximizer $h^*(t,i)$ of $V(t,i)$ in \eqref{eq:fixed point}, i.e.
\begin{equation}\label{eq:fixed_point_optimal}
            V(t,i) = U\big(g(i)-c(t+ h^*(t,i)) \big)  e^{-q_i\; h^*(t,i)} + \int_0^{h^*(t,i)}e^{-q_i s} \sum_{j\neq i} q_{ij} V(t+s,j) \;ds.
          \end{equation}
\end{theorem}

\begin{proof}
Since $V_n$ is increasing in $n$ we can define $V_\infty(t,i) := \lim_{n\to\infty} V_n(t,i)$. Obviously we have $V(t,i) \ge V_n(t,i)$ for all $n\in\N$ and thus $V(t,i) \ge V_\infty(t,i)$. On the other hand we have by assumption (A2) for any $\tau\in\Sigma$:
$$ V_\infty(t,i) = \liminf_{n\to\infty}V_n(t,i) \ge  \liminf_{n\to\infty}V_n(t,i,\tau) \ge V(t,i,\tau)$$
which implies that $V_\infty(t,i) \ge V(t,i)$. Altogether we have shown that $V=V_\infty$. Next observe that the $\mathrm{T}$-operator is monotone, i.e. $v\le w$ implies that $\mathrm{T}v\le \mathrm{T}w$. Hence with Theorem \ref{thm:bellman} we obtain that for all $k\in\N$:
 \begin{equation*}
            V_{k+1}(t,i)    =  (\mathrm{T}V_k)(t,i) \le (\mathrm{T}V)(t,i)
                      \end{equation*}
which implies that $V(t,i) \le  (\mathrm{T}V)(t,i)$. Now on the other hand we have
\begin{eqnarray*}
            V_{k+1}(t,i) & = &  (\mathrm{T}V_k)(t,i)\\
             &\ge & U\big(g(i)-ct-c\vartheta \big)  e^{-q_i \vartheta} + \int_0^{\vartheta}e^{-q_i s} \sum_{j\neq i} q_{ij} V_k(t+s,j) \;ds
\end{eqnarray*}
for all $\vartheta\ge 0$. Hence we obtain with monotone convergence that
\begin{eqnarray*}
            V(t,i) &=&  \lim_{k\to\infty} V_{k+1}(t,i)\\
            &\ge & U\big(g(i)-ct-c\vartheta \big)  e^{-q_i\vartheta} + \int_0^{\vartheta}e^{-q_i s} \sum_{j\neq i} q_{ij}\lim_{k\to\infty}  V_k(t+s,j) \;ds\\
            &= &  U\big(g(i)-ct-c\vartheta \big)  e^{-q_i\vartheta} + \int_0^{\vartheta}e^{-q_i s} \sum_{j\neq i} q_{ij} V(t+s,j) \;ds.
\end{eqnarray*}
for all $\vartheta\ge 0$. Thus it follows that $V(t,i) \ge (\mathrm{T}V)(t,i)$ which implies the statement.

The existence of a maximizer follows from Lemma \ref{lem:M1} and from Lemma \ref{lem:M2} respectively, since in case $U$ is defined on $\R$ we get that
\begin{eqnarray*}
  V  &\in & \big\{ v\in \mathbb{M} \; |\; v(t,i) \le a\sup_{\tau \in \Sigma^M} \Eop_i\big[g^+(X_{\tau})-c(t+\tau)\big] +b,\\
   &&  \quad\quad t\mapsto v(t,i) \mbox{ is decreasing  for all } i\in S\big\}.
\end{eqnarray*}
and in case $U$ is defined on a subset $[d,\infty)$
\begin{eqnarray*}
  V &\in& \big\{ v(\cdot,i): [0,d(i)] \to \R \mbox{ measurable for } i\in S\; |\; v(t,i) \le a\sup_{\tau \in \Sigma^M} \Eop_i\big[g^+(X_{\tau})-c(t+\tau)\big] +b,\\
   &&\quad\quad \; t\mapsto v(t,i) \mbox{ is decreasing for all } i\in S\big\}.
\end{eqnarray*}
The case that $U$ is defined on a subset $(d,\infty)$ is similar.
\end{proof}

An optimal stopping time is now obtained as follows:

\begin{theorem}\label{thm:tau*opt}
Let $h^*$ be the maximizer in \eqref{eq:fixed point} and suppose $\tau^* := (h^*,h^*,\ldots)$ satisfies \linebreak $\Pop_i(\tau^*<\infty)=1$ for $i\in S$ and
\begin{equation}\label{eq:optimal-stopping-time-unrestricted-case-condition}
    \lim_{n\to\infty} \Eop_i\Big[V(S_n+t,Z_n) 1_{\{\tau^* \ge S_n\}}\Big] = 0.
  \end{equation}
Then $\tau^*$ is an optimal stopping time for problem \eqref{eq:problem}.
\end{theorem}

\begin{proof}
First note that it can be shown by induction on $n$ that
\begin{align*}
    V(t,i) &= (\mathrm{T}^nV)(t,i)\\
           &= \Eop_i\left[ U\left(g\big(X_{\tau^*}\big) -c(t+\tau^*) \right) 1_{\{\tau^*<S_n\}}
                        +V(S_n +t,Z_n) 1_{\{\tau^*\ge S_n\}} \right].
  \end{align*}
For $n=1$ this is exactly \eqref{eq:fixed_point_optimal}:
\begin{equation*}
            V(t,i) = U\big(g(i)-c(t+h^*(t,i)) \big)  e^{-q_i\; h^*(t,i)} + \int_0^{h^*(t,i)}e^{-q_i s} \sum_{j\neq i} q_{ij} V(t+s,j) \;ds.
\end{equation*}
The induction step follows like in the proof of Theorem \ref{thm:reward-iteration}.

Since $S_n\to\infty$ $\Pop_i$-a.s.\ and since $\Pop_i(\tau^*<\infty)=1$, we know that
$1_{\{\tau^* < S_n\}} \uparrow 1 \quad\P_i\text{--a.s.}$.  An application of the monotone convergence theorem as well as condition \eqref{eq:optimal-stopping-time-unrestricted-case-condition} yields
  \begin{align*}
    V(t,i) &= \lim_{n\to\infty} \Eop_i\left[ U\left(g\big(X_{\tau^*}\big) -ct-c\tau^* \right) 1_{\{\tau^*<S_n\}}
                        +V(S_n +t,Z_n) 1_{\{\tau^*\ge S_n\}} \right]\\
            &= \Eop_i\left[ U\left(g\big(X_{\tau^*}\big) -ct-c\tau^* \right)\right].
  \end{align*}
  Hence, $\tau^\star\in\Sigma^M$ fulfills $ V(t,i) = \sup_{\tau\in\Sigma^M} V(t,i,\tau) = V(t,i,\tau^*)$
  and thus is an optimal stopping time for problem \eqref{eq:problem}.
\end{proof}

\begin{remark}\label{rem:consist}
The consistency condition formulated in Lemma \ref{lem:consist} also holds for the maximizer $h^*$ in \eqref{eq:fixed point}.
\end{remark}

\subsection{Exponential utility}
Let us now again consider the special case $U(x)=-e^{-\gamma x}$ for $\gamma>0$ and $x\in\R$.  From the finite horizon case we obtain here that $V(t,i)=e^{c\gamma t} W(i)$ with a function $W:S\to (-\infty,0]$. Here we obtain

\begin{theorem}\label{thm:exp1}  If $q_i\le c\gamma$ we obtain that $W(i)= -e^{-\gamma g(i)}$ and $\tau^*\equiv 0$ is optimal. Now suppose that $q_i > c\gamma$. Then it holds:
  \begin{itemize}
  \item[a)]  For  $(t,i)\in [0,\infty)\times S$ it holds that
          \begin{align}\label{eq:bellman-equation-exp}
            W(i) =& \max\bigg\{-e^{-\gamma g(i)}, \sum_{j\neq i} \frac{q_{ij}}{q_i-c\gamma} W(j) \bigg\}.
          \end{align}
  \item[b)] There exists a maximizer $f^*:S\to \{0,\infty\}$ (where we set $f^*(i):=0$ if the maximum in \eqref{eq:bellman-equation-exp} is attained in the first expression and $f^*(i):=\infty$ if the maximum is attained in the second expression). Let
      \begin{equation*}
   \tau^* = f^*(Z_0) 1_{\{\tau^* < S_1\}} + \sum_{k=1}^{\infty} (f^*(Z_k)+S_k) 1_{\{S_k\le \tau^* < S_{k+1}\}},\quad \Pop-a.s.
\end{equation*}
       and suppose that for $i\in S$,  $\Pop_i(\tau^*<\infty)=1$ and
\begin{equation*}\label{eq:optimal-stopping-time-unrestricted-case-condition}
    \lim_{n\to\infty} \Eop_i\Big[W(Z_n) 1_{\{\tau^* \ge S_n\}}\Big] = 0.
  \end{equation*}
Then $\tau^*$ is the optimal stopping time for the infinite horizon stopping problem with exponential utility.
  \end{itemize}
\end{theorem}

\begin{proof}
If $q_i\le c\gamma$ we have $V_k(t,i)=-e^{c\gamma t} e^{-\gamma g(i)}$. Thus,
$$V(t,i)= \lim_{k\to\infty} V_k(t,i) = -e^{c\gamma t} e^{-\gamma g(i)} = e^{c\gamma t} W(i)$$
and $\tau^*\equiv 0$ is optimal. If $q_i>c\gamma$ we plug $V(t,i)= e^{c\gamma t} W(i)$ into \eqref{eq:fixed point} to obtain \eqref{eq:bellman-equation-exp}. The remainder then follows from Theorem \ref{thm:tau*opt}.
\end{proof}

\section{Sufficient Conditions for Immediate Stopping}
In this section we will give sufficient conditions for utility functions $U$ with $dom(U)=\R$ which are continuously differentiable, under which it is optimal to stop directly after a jump. For this purpose define
for $t\ge 0$:

\begin{align}
      S^0_t := \Big\{ i\in S \; \Big|\; &\sum_{j\neq i} \frac{q_{ij}}{q_i} U\big(g(j)-c\vartheta\big) \le U\big(g(i)-c\vartheta\big) + \frac{c}{q_i} U'\big(g(i)-c\vartheta\big) \quad\text{ for all } \vartheta\ge t \Big\}.\label{eq:S-infty-set}
\end{align}
Note that for $0\le t\le t'$ we obtain $S^0_t \subseteq S^0_{t'}$. It holds:

\begin{theorem}\label{thm:OLA}
 Suppose that for all $t\ge 0$, $i\in S^0_t$ and $j\neq i$ the implication
  \begin{equation}\label{eq:closure-assumption}
    q_{ij}\neq 0 \Longrightarrow j\in S^0_{t}
  \end{equation}
  is valid. Then the maximizer in \eqref{eq:fixed point} is given by $h^*(t',i) = 0$ for every $t'\ge t$ and $V(t,i)=U\big(g(i)-ct\big)$ for all $i\in S^0$.
\end{theorem}

\begin{proof}
 We will prove by induction on $k\in\N$, that for any $t\ge 0$ and $i\in S^0_t$, the equality $V_k(t,i) = V_k(t,i,0) = U(g(i)-ct)$
  holds and the optimal stopping time is given by $\tau^*\equiv 0$. This then implies that $V(t,i)=U\big(g(i)-ct\big)$.

To this end, fix $t\ge 0$, $i\in S^0_t$ and suppose that $n=1$.  Consider
$$ m(t,\vartheta):= U\big(g(i)-c(t+\vartheta) \big)  e^{-q_i\vartheta} + \int_0^{\vartheta}e^{-q_i s} \sum_{j\neq i} q_{ij} U\big(g(j)-c(t+s)\big) \;ds. $$
Obviously $m$ is differentiable w.r.t.\ $\vartheta$ and we obtain
\begin{eqnarray*}
  &&\frac{\partial}{\partial\vartheta} m(t,\vartheta) \le  0 \quad \Leftrightarrow \\
   &&  \sum_{j\neq i} \frac{q_{ij}}{q_i} U\big(g(j)-c(t+\vartheta)\big) \le U\big(g(i)-c(t+\vartheta)\big) + \frac{c}{q_i} U'\big(g(i)-c(t+\vartheta)\big) .
\end{eqnarray*}
Due to our assumption \eqref{eq:closure-assumption} this inequality is satisfied for all $\vartheta\ge 0$. Hence the maximizer in \eqref{eq:bellman-equation} for $k=0$ satisfies $h_1^*(t,i)=0$ and we obtain $V_1(t,i)=U\big(g(i)-ct\big)$. Now suppose the statement is true for $k\in\N$. The optimality equation in \eqref{eq:bellman-equation} then reads
 \begin{align}
V_{k+1}(t,i) =& \sup_{\vartheta\ge 0}\bigg\{ U\big(g(i)-ct-c\vartheta \big)  e^{-q_i\vartheta} + \int_0^{\vartheta}e^{-q_i s} \sum_{j\neq i} q_{ij} V_k(t+s,j) \;ds \bigg\}\\
=& \sup_{\vartheta\ge 0}\bigg\{ U\big(g(i)-ct-c\vartheta \big)  e^{-q_i\vartheta} + \int_0^{\vartheta}e^{-q_i s} \sum_{j\neq i} q_{ij} U\big(g(j)-c(t+s)\big) \;ds \bigg\}
\end{align}
which is again the same problem and we obtain $V_{k+1}(t,i)=U\big(g(i)-ct\big)$ and $h_k^*(t,i)=0$. Altogether we have shown the result.
\end{proof}

\begin{remark}
Instead of $S^0_t$ we can also consider
\begin{align*}
      S^\infty_t := \Big\{ i\in S \; \Big|\; &\sum_{j\neq i} \frac{q_{ij}}{q_i} U\big(g(j)-c\vartheta\big) > U\big(g(i)-c\vartheta\big) + \frac{c}{q_i} U'\big(g(i)-c\vartheta\big) \quad\text{ for all } \vartheta\ge t \Big\}.
\end{align*}
where we reverse the inequality. In this case we obtain for $i\in S_t^\infty$:
\begin{eqnarray*}
  &&\frac{\partial}{\partial\vartheta} m(t,\vartheta) \ge  0 \quad \Leftrightarrow  \\
   && \sum_{j\neq i} \frac{q_{ij}}{q_i} V\big(g(j)-c(t+\vartheta)\big) \ge \sum_{j\neq i} \frac{q_{ij}}{q_i} U\big(g(j)-c(t+\vartheta)\big) \\
   && > U\big(g(i)-c(t+\vartheta)\big) + \frac{c}{q_i} U'\big(g(i)-c(t+\vartheta)\big)
\end{eqnarray*}
which implies that $h^*(t',i)=\infty$ for every $t'\ge t$, i.e. we will never stop as long as the Markov chain is in a state $i\in S_t^\infty$.
\end{remark}

\subsection{Exponential utility}
Let us now again consider the special case $U(x)=-e^{-\gamma x}$ for $\gamma>0$ and $x\in\R$. In this case the set  $S^0_t$ is independent of $t$ and given by:
\begin{align}
      S^0 :=\Big\{ i\in S \; |\; \sum_{j\neq i} \frac{q_{ij}}{q_i-c\gamma} e^{-\gamma g(j)} \ge e^{-\gamma g(i)} \text{ and } q_i>c\gamma
                 \quad\text{ or }\quad q_i\le c\gamma \Big\}.\label{eq:S-OLA-set}
    \end{align}
Here we obtain:
\begin{theorem}
 Suppose that for all $i\in S^0$ and $j\neq i$ the implication
  \begin{equation}\label{eq:closure-assumption-exp}
    q_{ij}\neq 0 \Longrightarrow j\in S^0
  \end{equation}
  is valid. Then the maximizer in  \eqref{eq:bellman-equation-exp} is for $i\in S^0$ given by $f^*(i) = 0$ and $W(i)=-e^{-\gamma g(i)}$. If $i\notin S^0$ then $f^*(i)=\infty$. Thus it holds that $$\tau^* = \inf\{ t\ge 0\; |\; X_t \in S^0\}.$$
\end{theorem}

\begin{proof}
For $i\in S^0$ the statement follows from Theorem \ref{thm:OLA}. For $i\notin S^0$ observe that
$$ -e^{-\gamma g(i)} \le \sum_{j\neq i} \frac{q_{ij}}{q_i-c\gamma} (-e^{-\gamma g(j)}) \le \sum_{j\neq i} \frac{q_{ij}}{q_i-c\gamma}W(j)$$
which implies $f^*(i)=\infty$ due to Theorem \ref{thm:exp1}.
\end{proof}

\begin{remark}
Note that the set $S^0$ in this case can be written in a different way as
$$ S^0 =\Big\{ i\in S \; |\; \Eop_i[-e^{c\gamma S_1-\gamma g(Z_1)}] \le -e^{-\gamma g(x)} \text{ and } q_i>c\gamma
                 \quad\text{ or }\quad q_i\le c\gamma \Big\}$$
which means that we compare the expected utility we obtain when we stop immediately with the expected utility we obtain when we stop after the next jump time point of the Markov chain. In this case we can interpret the stopping rule as a 'one-step look ahead rule'.
\end{remark}

\begin{example}
Suppose that $(X_t)$ is a homogeneous Poisson process with intensity $\lambda>0$, i.e. $q_{ii+1}=\lambda$ and that $U(x)=-e^{-\gamma x}$. In this case
$$   S^0 =\Big\{ i\in \N_0 \; |\; \frac{\lambda}{\lambda-c\gamma} e^{-\gamma g(i+1)} \ge e^{-\gamma g(i)} \text{ and } \lambda >c\gamma
                 \quad\text{ or }\quad \lambda\le c\gamma \Big\}.$$
Thus, if $  \lambda\le c\gamma$ we have $S^0=\N_0$ which obviously satisfies condition \eqref{eq:closure-assumption-exp} and if $\lambda>c\gamma$ we have
$$   S^0 =\Big\{ i\in \N_0 \; |\; g(i+1)-g(i) \le \frac{1}{\gamma} \ln\Big(\frac{\lambda}{\lambda-c\gamma}\Big) \Big\}.$$
In case $g$ is concave we have that $i\mapsto g(i+1)-g(i)$ is decreasing and the set $S^0$ is of the form $S^0=\{\bar{i},\bar{i}+1,\ldots\}$ with $\bar{i}\in\N_0$ which satisfies \eqref{eq:closure-assumption-exp}. Hence it follows that the optimal stopping time is given by $\tau^* = \inf\{t\ge 0 \; |\; X_t \ge \bar{i}\}$.
\end{example}

\section{Influence of risk aversion}
Finally we discuss the influence of the risk attitude of the decision maker on the optimal stopping time. We use the Arrow-Pratt function of absolute risk aversion \eqref{BReq:AP} to measure the risk sensitivity and concentrate on utility functions $U$ which are defined on $\R$ and $U\in C^2(\R)$. A utility function $U$ is said to be {\em more risk averse} than a utility function $W$ if $l_U(x) \ge l_W(x)$ for all $x\in\R$. For our purpose it is crucial to note that a utility function $U$ is more risk averse than a utility function  $W$ if and only if, there exits an increasing concave function $r:\R\to\R$ such that $U=r\circ W$. In what follows we denote all quantities which refer to utility function $U$ by $h^*(t,i,U), V(t,i,U)$ etc. and similar for $W$.

\begin{theorem}\label{BRtheo:compstoppingrule}
Suppose that the utility function $U$ is more risk averse than the utility function  $W$. For all states $(t,i)\in\R_+\times S$ we obtain that $h^*(t,i,W) =0 $ implies $h^*(t,i,U)=0$, i.e. a more risk-averse decision maker will not stop later.
\end{theorem}

\begin{proof}
Let $r$ be such that $U=r\circ W$. We first prove by induction on $k$ that $V_k(t,i,U)\le r\circ V_k(t,i,W)$ for all $(t,i)$ and $k$.
First for $k=0$ we have
\begin{equation*}
    V_0(t,i,U) = U\big( g(i)-ct\big) = r\circ W\big( g(i)-ct\big) = r\circ V_0(t,i,W).
\end{equation*}
Using the Jensen inequality, the induction hypothesis and the fact that $r$ is increasing and concave we obtain
\begin{eqnarray*}
  &&V_{k+1}(t,i,U) =\\
  &=&  \sup_{\vartheta\ge 0}\bigg\{ U\big(g(i)-c(t+\vartheta) \big)  e^{-q_i\vartheta} + \int_0^{\vartheta}e^{-q_i s} \sum_{j\neq i} q_{ij} V_k(t+s,j,U) \;ds \bigg\}\\
   &\le& \sup_{\vartheta\ge 0}\bigg\{ r\circ W\big(g(i)-c(t+\vartheta) \big)  e^{-q_i\vartheta} + \int_0^{\vartheta}q_i e^{-q_i s} \sum_{j\neq i} \frac{q_{ij}}{q_i} r\circ V_k(t+s,j,W) \;ds \bigg\}\\
 &\le& \sup_{\vartheta\ge 0}\bigg\{ r\circ W\big(g(i)-c(t+\vartheta) \big)  e^{-q_i\vartheta} + r\circ \Big(\int_0^{\vartheta}q_i e^{-q_i s} \sum_{j\neq i} \frac{q_{ij}}{q_i} V_k(t+s,j,W) \;ds\Big) \bigg\}\\
      &=& r\circ \sup_{\vartheta\ge 0}\bigg\{ W\big(g(i)-c(t+\vartheta) \big)  e^{-q_i\vartheta} + \int_0^{\vartheta}q_i e^{-q_i s} \sum_{j\neq i} \frac{q_{ij}}{q_i} V_k(t+s,j,W) \;ds \bigg\} \\
   &=& r\circ  V_{k+1}(t,i,W).
\end{eqnarray*}
Letting $k\to\infty$ this yields $V(t,i,U)\le r \circ V(t,i,W)$.
This implies in particular that the inequality
\begin{eqnarray*}
  W\big(g(i)-ct\big) &\ge& \sup_{\vartheta\ge 0}\bigg\{ W\big(g(i)-ct-c\vartheta \big)  e^{-q_i\vartheta} + \int_0^{\vartheta}q_i e^{-q_i s} \sum_{j\neq i} \frac{q_{ij}}{q_i} V(t+s,j,W) \;ds \bigg\}
  \end{eqnarray*}
leads to
\begin{eqnarray*}
&& U\big(g(i)-ct\big)=  r\circ W\big(g(i)-ct\big) \\
&\ge& r\circ \sup_{\vartheta\ge 0}\bigg\{ W\big(g(i)-ct-c\vartheta \big)  e^{-q_i\vartheta} + \int_0^{\vartheta}q_i e^{-q_i s} \sum_{j\neq i} \frac{q_{ij}}{q_i} V(t+s,j,W) \;ds \bigg\}  \\
&=& r \circ V(t,i,W)\\
&\ge & V(t,i,U)\\
    &=& \sup_{\vartheta\ge 0}\bigg\{ U\big(g(i)-ct-c\vartheta \big)  e^{-q_i\vartheta} + \int_0^{\vartheta}e^{-q_i s} \sum_{j\neq i} q_{ij} V(t+s,j,U) \;ds \bigg\}
\end{eqnarray*}
By definition this means that $h^*(t,i,W)=0$ implies $h^*(t,i,U)=0$. Thus the statement follows.
\end{proof}

\begin{remark}
From Lemma \ref{lem:consist} it also follows that $h^*(t,i,W)\ge h^*(t,i,U)$ for all $t\ge 0$:  Suppose $h^*(t,i,W)=\delta >0$ and assume $h^*(t,i,U)>\delta$. This implies according to Remark \ref{rem:consist} that $h^*(t+\delta,i,W)=0$ and $h^*(t+\delta,i,W)=h^*(t,i,U)-\delta >0$ which is a contradiction to Theorem \ref{BRtheo:compstoppingrule} then. Thus we have $\tau^*(W) \ge_{st} \tau^*(U)$.
\end{remark}

\section{Risk-sensitive House selling problem}
In this section we consider the classical house selling problem in a continuous-time Markov chain setting. In order to have a reasonable model we consider the following special process $(X_t)$: Let $S=\{1,\ldots,m\}$ and the intensity matrix of $(X_t)$ be given by
$$ Q = \left( \begin{array}{cccc}
    -q_1 & \alpha_1 & \ldots & \alpha_m\\
    \alpha_1 & -q_2 & \ldots & \alpha_m\\
    \vdots & & \ddots &\\
    \alpha_1 & \ldots & \alpha_{m-1} & -q_m
    \end{array}\right)$$
with $\alpha_i > 0$ for all $i$ and $q_i := \sum_{j\neq i} \alpha_j$. Set $\alpha:= \sum_{j=1}^m \alpha_j$. Using the well-known uniformization technique (see e.g. \cite{cin75}) it follows immediately that $(X_t)$ is in distribution equal to the process $(\hat{X}_t)$ with
$$ \hat{X}_t = \sum_{k=0}^\infty Z_k \cdot 1_{\{S_k \le t<S_{k+1}\}}$$
where $Z_0=X_0$ and $Z_1, Z_2,\ldots$ are independent and identically distributed random variables with $\Pop(Z_k=i) = \frac{\alpha_i}{\alpha}$ and the random variables $S_1-S_0, S_2-S_1,\ldots$ are also independent and identically distributed random variables with $S_1-S_0\sim \exp(\alpha)$. The interpretation of $(X_t)$ is as follows: Suppose we want to sell a house. After an exponentially distributed amount of time a new offer for the house arrives. Offers are independent and identically distributed like $Z_1$. As long as the house is not sold, we have to pay for maintenance at rate $c>0$. Suppose $U$ is defined on $\R$ and let us consider the infinite horizon problem. We set $g(i)=i$. Thus, the problem is given by
\begin{equation}\label{eq:houseproblem} \sup_{\tau \in \Sigma^M} \Eop_i\big[ U\big(X_\tau-c\tau\big)\big].\end{equation}
The optimality equation applied to the uniformized model reads:
$$  V(t,i) =\sup_{\vartheta\ge 0}\bigg\{ U\big(i-c(t+\vartheta) \big)  e^{-\alpha\vartheta} + \int_0^{\vartheta}e^{-\alpha s} \sum_{j} \alpha_{j} V(t+s,j) \;ds \bigg\}.$$
As before we denote the maximizer by $h^*(t,i)$. Let us first make the following simple observation: $V(t,i)\le U(m-ct)$ and $h^*(t,m)\equiv 0$. The maximal reward for the house we can obtain is $m$. Thus, we will stop immediately when an offer of $m$ arrives, otherwise we can only get less. This explains the inequality. With this observation we obtain:

\begin{lemma}
The stopping time $\tau^* := (h^*,h^*,\ldots)$ satisfies $\Pop_i(\tau^*<\infty)=1$ for $i\in S$ and
\begin{equation}\label{eq:houselimit}
    \lim_{n\to\infty} \Eop_i\Big[V(S_n+t,Z_n) 1_{\{\tau^* \ge S_n\}}\Big] = 0.
  \end{equation}
Thus $\tau^*$ is an optimal stopping time for problem \eqref{eq:houseproblem}.
\end{lemma}

\begin{proof}
Since $h^*(t,m)\equiv 0$ we obtain that $\tau^*\le_{st} \tau^m := \inf\{ t\ge 0 : X_t = m\}$ where $\le_{st}$ is the usual stochastic order. Since $(X_t)$ is positive recurrent we have that $\Pop_i(\tau^m <\infty)=1$ for all $i\in S$. Thus, the same is true for $\tau^*$. Since $V$ is bounded by $U(1-ct) \le V(t,i) \le U(m-ct)$ and $1_{\{\tau^* < S_n\}} \uparrow 1 \quad\P_i\text{--a.s.}$ equation \eqref{eq:houselimit} follows. Thus, the statement is a consequence of Theorem \ref{thm:tau*opt}.
\end{proof}

The next lemma further explores the structure of the optimal stopping time.

\begin{lemma}
For $i=1,\ldots,m-1$ and $t\ge 0$ we have that $h^*(t,i+1) \le h^*(t,i)$, i.e. the larger the offer, the earlier we will stop.
\end{lemma}

\begin{proof}
Let us define
$$ m(t,i,\vartheta) := U\big(i-c(t+\vartheta) \big)  e^{-\alpha\vartheta} + \int_0^{\vartheta}e^{-\alpha s} \sum_{j} \alpha_{j} V(t+s,j) \;ds.$$
By definition we have $h^*(t,i) = argmax_{\vartheta\ge 0} m(t,i,\vartheta)$. Now obviously
$$m(t,i+1,\vartheta) = m(t,i,\vartheta)+ \big( m(t,i+1,\vartheta)-m(t,i,\vartheta)\big)$$
where
$$ m(t,i+1,\vartheta)-m(t,i,\vartheta) = e^{-\alpha \vartheta}\big(U(i+1-c(t+\vartheta))-U(i-c(t+\vartheta) \big).$$
Since $U$ is concave $i\mapsto m(t,i+1,\vartheta)-m(t,i,\vartheta)$ is decreasing. Thus, the maximum point $h^*(t,i+1)$ of $\vartheta\mapsto m(t,i+1,\vartheta)$ has to satisfy $h^*(t,i+1) \le h^*(t,i)$.
\end{proof}

\bibliographystyle{plainnat}

\end{document}